\documentclass[10pt]{amsart}
\usepackage{anysize, amsfonts, amsthm, amssymb, amsmath, bbm,mathtools,tikz, hyperref,cite, verbatim}
\marginsize{.85in}{.85in}{.85in}{.85in}

\newcommand{\mtd}{\raise.17ex\hbox{$\scriptstyle\mathtt{\sim}$}}
\newcommand{\nmtd}{\raise.17ex\hbox{$\scriptstyle\mathtt{\not\sim}$}}
\newcommand{\lold}{\lambda}
\newcommand{\lnew}{\lambda'}
\DeclareMathOperator{\VOL}{Vol}
\newcommand{\vol}[1]{\VOL\left(#1\right)}

\newcommand{\diam}[1]{\DIAM\left(#1\right)}
\renewcommand{\diam}[1]{D_{#1}}

\newcommand{\R}{\mathbb{R}}
\renewcommand{\L}{\mathcal{L}}
\newcommand{\one}{\ensuremath{\mathbbm{1}}}
\newcommand{\bigOh}[1]{\ensuremath{\mathcal{O}\!\left( #1 \right)}}
\newcommand{\bigTheta}[1]{\ensuremath{\mathit{\theta}\!\left( #1 \right)}}
\newcommand{\ba}{\backslash}
\newcommand{\eps}{\epsilon}

\newtheorem{thm}{Theorem}[section]
\newtheorem{lemma}[thm]{Lemma}
\newtheorem{cor}[thm]{Corollary}
\theoremstyle{definition}

\newtheorem*{exa}{Example}

\title{Bounds on Geometric Eigenvalues of Graphs}
\author{Mary Radcliffe}
\author{Christopher Williamson}
\thanks{M. Radcliffe, University of Washington, Seattle}
\thanks{C. Williamson, The Chinese University of Hong Kong, supported by Hong Kong RGC GRF grant CUHK410112}

\begin{document}
\maketitle

\begin{abstract}
The smallest nonzero eigenvalue of the normalized Laplacian matrix of a graph has been extensively studied and shown to have many connections to properties of the graph. We here study a generalization of this eigenvalue, denoted $\lambda(G, X)$, introduced by Mendel and Naor, obtained by embedding the vertices of the graph $G$ into a metric space $X$. We consider general bounds on $\lambda(G,X)$ and on $\lambda(G, H)$, where $H$ is a graph under the standard distance metric, generalizing some existing results for the standard eigenvalue. We consider how $\lambda(G, H)$ is affected by changes to $G$ or $H$, and show that $\lambda$ is not monotonic in either $G$ or $H$.

\end{abstract}

\section{Introduction}

The use of eigenvalues to study graphs has a long history in graph theory. Since at least the 1980s, the eigenvalues of various matrices have been used to study properties of a graph, including many connectivity features, distance and diameter properties, automorphisms, random walks, and a litany of graph invariants. Results involving spectra of a graph have been catalogued in many surveys and books, such as \cite{beineke2004topics, chung1997spectral, cvetkovic1988recent, cvetkovic1990largest, spielman2007spectral, spielman2010algorithms}, for example.

Of particular interest in the study of graph theory is the first nonzero eigenvalue of the normalized Laplacian matrix. This single quantity has ties to connectivity, the rate of convergence of a random walk over the graph, the diameter, discrepancy bounds on the number of edges between sets, and many other important properties. In addition, the first nonzero eigenvalue is used to classify expander graphs, applications for which have been found in many facets of computer science, group theory, geometry, topology, and other areas. There are many surveys available on the properties of expanders and the first eigenvalue, such as \cite{hoory2006expander, lubotzky2010discrete, lubotzky2012expander}.

Recently, work has begun on generalizing the notion of the first nonzero eigenvalue of the normalized Laplacian matrix in geometric terms \cite{mendel2014expanders, mendel2013nonlinear, mendel2010towards, dumitriu2014nonlinear}. This is related to the study of the distortion of embeddings between metric spaces, and this perspective has also appeared in the literature; see, for example \cite{lee2007nonlinear, linial1995geometry, matouvsek1997embedding}. We here build upon this literature by studying this embedding constant in a general setting.

To begin, let us examine the desired generalization. We start with the standard definition of the normalized Laplacian. Any notation not explicitly defined will be given in Section \ref{S:notation} below.

 For a given graph $G$, we define the adjacency matrix $A$ to be the $\{0,1\}$-valued matrix indexed by $V(G)$ such that $A_{uv} = 1$ if $u\sim v$ and $0$ otherwise. Define the diagonal degree matrix $D$ to have $D_{vv}$ equal to $d_v$. The normalized Laplacian matrix is defined to be $\L = I-D^{-1/2}AD^{-1/2}$, where we take the convention that if $D_{vv}=0$, then $D^{-1/2}_{vv}=0$. It is well-known that the smallest eigenvalue of $\L$ is $\lambda_0=0$, with corresponding eigenvector $D^{1/2}\one$. Hence, by the Courant-Fischer theorem, we have that $\lambda_1=\inf_{f\perp D^{1/2}\one}( f^T\L f)/(f^Tf)$.

One can view the vector $f$ as a function from $V(G)$ to $\R$, where $f(v)=f_v$. From this perspective, some basic manipulations provide the following equivalent form for $\lambda_1$ (see, for example, \cite{chung1997spectral}):
\begin{equation}\label{E:defstandard} \lambda_1 = \inf_{f:V(G)\to \R} \frac{\vol{G}\sum_{u\sim v} |f(u)-f(v)|^2}{\sum_{u, v} |f(u)-f(v)|^2d_ud_v}.\end{equation}
Hence, one can view $\lambda_1$ as an attempt to compare the average distance between the embedding values at adjacent vertices to the average distance between the embedding values of an arbitrary pair of vertices. Roughly speaking, a small value of $\lambda_1$ indicates that adjacent vertices can be mapped quite close together, even as the vertices themselves are spread out. Intuitively (and actually) this would indicate poor connectivity of $G$, with the extreme case that $\lambda_1=0$ indicating that the graph is in fact disconnected.

In \cite{mendel2010towards}, the following geometrically based generalization was proposed. In Equation (\ref{E:defstandard}), one can view the quantity $|f(u)-f(v)|^2$ in terms of the distance between $f(u)$, and $f(v)$; that is, $|f(u)-f(v)|^2 = \ell_2(f(u), f(v))^2$. Hence, we can extend this definition to an arbitrary metric space $(X, d)$ by replacing $\ell_2(f(u), f(v))^2$ by $d(f(u), f(v))^2$, and taking the infimum over all functions from $V(G)$ to $X$. Specifically, we define
\begin{equation}\label{E:defgeo}
\lambda(G, X) = \inf_{f:V(G)\to X} \frac{\vol{G}\sum_{u\sim v} d(f(u),f(v))^2}{\sum_{u, v}d(f(u),f(v))^2d_ud_v}.
\end{equation}
Previous work on this constant has primarily been focused on regular graphs, and more specifically random regular graphs, and the ties between $\lambda(G, X)$ and expansion in a graph \cite{mendel2014expanders, mendel2013nonlinear, mendel2010towards, dumitriu2014nonlinear}. We here provide bounds on $\lambda(G, X)$ in the case that the metric space $X$ is itself a graph under the standard distance metric, and provide analogs to some classical theorems in spectral graph theory in this case. Specifically, we prove the following analogs to standard results in spectral graph theory.

\begin{thm}\label{T:mainone}
Let $G$ be a graph on $n$ vertices and $H$ be a graph. Then 
\begin{itemize}
\item $\lambda(G, H)=0$ if and only if $G$ is disconnected.
\item $\lambda(G, H)\leq \frac{n}{n-1}$, and equality is achieved if and only if $G=K_n$.
\end{itemize}
\end{thm} We also show that for a given graph family, if $\lambda\to 0$, we must have that $\lambda \in \bigOh{\frac{1}{n^2}}$, and show by example that this bound is asymptotically tight. We also provide some general bounds on the constant $\lambda(G, H)$ as a part of the proof of Theorem \ref{T:mainone}.

In \cite{mendel2010towards}, it is noted that for any metric space, $\lambda(G, X)\lesssim \sqrt{\lambda(G, \R)}$. We prove here a lower bound for $\lambda(G, X)$ in terms of $\lambda(G, \R)$ when $X$ is finite, namely the following.

\begin{thm}\label{T:relatetoR}
Let $X$ be a finite metric space. Then there exists an absolute constant $C$ such that for every connected graph $G$, 
\[\frac{C}{\log^2|X|}\lambda(G, \R)\leq \lambda(G, X).\]
\end{thm}

Finally, we consider how modifications to the graphs $G$ or $H$ can impact $\lambda(G, H)$. We provide examples showing that adding an edge to $G$ can both increase and decrease the value of $\lambda$ when $H$ is held constant, so that $\lambda$ is not monotone in $G$, and provide bounds on the ratio of the two eigenvalues. Similarly, we provide examples for which taking $H'$ a subgraph of $H$ also increases and decreases the value of $\lambda$ when $G$ is hold constant, so that $\lambda$ is also not monotone in $H$. However, we do have the following theorem:

\begin{thm}\label{T:ktwoextreme}Let $H$ be a connected graph on $k>1$ vertices. Then $\lambda(G, H)\leq \lambda(G, K_2)$.
\end{thm}
Hence, the single edge provides an extreme case for calculating $\lambda$.

To avoid confusion, throughout the remainder of this work, we shall refer to the classical first eigenvalue $\lambda_1$ as $\lambda(G, \R)$. We also take any graphs used as a metric space as connected, as otherwise the ratio $\frac{\vol{G}\sum_{u\sim v} d(f(u),f(v))^2}{\sum_{u, v}d(f(u),f(v))^2d_ud_v}$ may be undefined. Since $\lambda(G,H)$ for a disconnected graph $H$ is equal to the minimum of $\lambda(G,H_j)$ for connected components $H_j$ of $H$, it suffices to assume $H$ is connected.

\section{Notation}\label{S:notation}

Throughout, we shall use standard graph theoretic notation, as follows. 

For $G$ a graph, let $V(G)$ denote the vertex set of $G$, and $E(G)\subset {V(G)\choose 2}$ denote the edge set of $G$. We write $|V(G)|=n$, $|E(G)|=m$. For a vertex $v\in V(G)$, let $d_v$ denote the degree of $v$ in $G$. If needed, for clarification we will use $d_v(G)$ to denote the degree in $G$. The maximum degree in $G$ will be denoted by $\Delta$, and the minimum degree by $\delta$. The distance between two vertices $d_G(u, v)$, is the number of edges in a shortest path between $u$ and $v$. The diameter of $G$ is the maximum distance between two vertices, and will be denoted by $D_G$. For a collection $S$ of vertices in $G$, write $\vol{S} = \sum_{u\in S} d_u$. For simplicity, we write $\vol{G}$ to denote $\vol{V(G)}$. For two sets of vertices $S,T\subset V(G)$, let $e(S, T)$ denote the number of edges incident to both $S$ and $T$. 

Throughout we will view graphs also as metric spaces, using the distance function defined above. More specifically, we will consider the quantity $\lambda(G, H)$, where $(H, d_H)$ is a metric space over a graph $H$.  We shall typically write $|V(H)|=k$.  As there are two graphs involved, for clarity we shall typically use letters $u, v$ to indicate vertices in $V(G)$ and $i, j$ to indicate vertices in $V(H)$.

The complete graph $G=K_n$ is the graph with edge set $E(G)={V(G)\choose 2}$, that is, all possible pairs of vertices are an edge in $K_n$. The complete bipartite graph $G=K_{n_1, n_2}$ has vertex set $V(G)=V_1\cup V_2$, where $|V_1|=n_1$, $|V_2|=n_2$, and $\{u, v\}\in E(G)$ if and only if one of $u, v$ is a member of $V_1$ and the other is a member of $V_2$. Given a graph $G$, we define the density of $G$ to be $\rho = \frac{m}{{n\choose 2}}$; that is, $\rho$ is the proportion of possible edges that are present in $G$.

To compute $\lambda(G, X)$, one must minimize the fraction given in Equation (\ref{E:defgeo}). For a given function $f:V(G)\to X$, define
\begin{equation}\label{E:Rf}
R_f(G, X) = \frac{\vol{G}\sum_{u\sim v} d(f(u),f(v))^2}{\sum_{u, v}d(f(u),f(v))^2d_ud_v},
\end{equation}
so that $\lambda(G, X)=\inf_{f:V(G)\to X} R_f(G, X)$. When the metric space and graph are understood, we write $R_f$ in place of $R_f(G, X)$, for simplicity. 

As the embedding constant $\lambda(G, X)$ is related to metric embeddings, we shall make use of Bourgain's Embedding Theorem \cite{bourgain1985lipschitz} to prove Theorem \ref{T:relatetoR}. Although this theorem takes many forms, the specific version we shall use is as follows (see, for example, \cite{Magen2006}).

\begin{thm}\label{T:Bourgain}
There exist constants $c, C$ such that, for all finite metric spaces $X$, there exists a function $g:X\to \R^K$, where $K = \bigTheta{\log^2|X|}$ such that, for all $x, y\in X$,
\[ (c\log |X|) d_X(x, y)\leq \|g(x)-g(y)\|_1\leq (C\log^2|X| )d_X(x, y).\]
\end{thm}

We note that the constants $c, C$ are independent of the metric space $X$. Let $\phi_K:\R^K\to \R$ be the projection $\phi_K:\R^K\to\R$ with $\phi_K(v)=\sum_{i=1}^K v_i$, and note that for any vectors $v, w\in \R^K$, we have
\[\frac{1}{K}\|v-w\|_1\leq \|v-w\|_\infty\leq |\phi_K(v)-\phi_K(w)|\leq \|v-w\|_1.\] We then have the following immediate corollary to Bourgain's Embedding Theorem:

\begin{cor}\label{C:Bourgain}There exist absolute constants $c, C$ such that, for all finite metric spaces $X$, there exists a function $f:X\to \R$ such that, for all $x, y\in X$,
\[(c/\log |X|) d_X(x, y)\leq |f(x)-f(y)|\leq (C\log^2|X| )d_X(x, y).\]
\end{cor}

\section{Bounds on $R_f$}

One useful tool to provide simplistic bounds on $\lambda(G, H)$ will be to bound $R_f$ simultaneously for all $f$. We present here some basic bounds that shall appear throughout the remainder of this work. We begin with the following optimization that will be useful in bounding the denominator of $R_f$.

\begin{lemma}\label{L:opt}
Let $x\in\mathbb{R}^k$, with $k\geq 2$, be a vector satisfying:\begin{enumerate}\item $\sum_ix_i=C\geq 6$. \item For all $i$, $x_i\geq 0$.\item  For all $i$, $x_i\in\mathbb{Z}$.\item There exist $i,j$ where $i\neq j$ such that $x_i,x_j>0$.\end{enumerate}Then, $\|x\|_2^2\leq C^2-2C+2$.
\end{lemma}

\begin{proof}
First, if $n=2$, this becomes an optimization problem in only one variable. If we set $x_1=j$, then we need only determine
\[\displaystyle\max_{\substack{j\in\mathbb{Z}\\1\leq j\leq C-1}}\left(j^2+(C-j)^2\right).\] Basic calculus shows that the maximum occurs at the endpoints of the interval, namely, where $j=1$ or $j=C-1$, obtaining a maximum value of $(C-1)^2+1 = C^2-2C+2$, as desired.

Now, let us suppose that $x\in \R^k$ has at least three nonzero entries, say $x_1, x_2, x_3$. Define $y=x-x_3e_3+x_3e_2$. Note that $y$ is also a feasible vector for the optimization, and that $\|y\|_2^2 = \|x\|_2^2 - (x_2^2+x_3^2)+(x_2+x_3)^2=\|x\|_2^2+2x_2x_3>\|x\|_2^2$. Hence, the optimum must occur at a vector with precisely 2 nonzero entries, and we may use the above argument for the case $n=2$ to obtain the desired result.
\end{proof}

We can immediately use this result to provide the following simple lower bound on the denominator in $R_f$.

\begin{thm}\label{denomlower}
Let $G$ and $H$ be connected graphs with $\vol{G}\geq 6$, $k=|V(H)|$, and $f:V(G)\rightarrow V(H)$ be an arbitrary non-constant function. Then, $\sum_{u,v}d(f(u),f(v))^2d_ud_v\geq \vol{G}-1$.
\end{thm}

\begin{proof}
For all $i\in V(H)$, let $S_i = f^{-1}(i) \subset V(G)$. Let $x_i=\vol{S_i}$, and let $x=(x_1, x_2, \dots, x_k)\in \R^k$. As $f$ is a nonconstant function, we have that $x$ is a feasible vector for the optimization problem in Lemma \ref{L:opt} with $C=\vol{G}$, and thus $\|x\|_2^2\leq \vol{G}^2-2\vol{G}+2$. Note that for any $i, j\in V(H)$, we have $\sum_{u\mapsto i, v\mapsto j} d(f(u), f(v))d_ud_v=d(i, j)\vol{S_i}\vol{S_j}$. As $d(i, j)\geq 1$ for all $i\neq j$, we obtain

\begin{eqnarray*}
\sum_{u,v}d(f(u),f(v))^2d_ud_v & = & \sum_{i\neq j} d(i, j)^2\vol{S_i}\vol{S_j}\\
& \geq & \sum_{i\neq j} x_ix_j\\
& = & \left(\frac{1}{2}\left(\sum_ix_i\right)^2-\frac{1}{2}\left(\sum_ix_i^2\right)\right)\\
& = & \frac{\|x\|_1^2-\|x\|_2^2}{2}\\
& \geq & \frac{\vol{G}^2-(\vol{G}^2-2\vol{G}+2)}{2} = \vol{G}-1.
\end{eqnarray*}\end{proof}

Similarly, as for all $x\in \R^k$ we have $\|x\|_2\geq\frac{1}{\sqrt{k}}\|x\|_1$, we have the following simple upper bound on the denominator in $R_f$.

\begin{thm}\label{denomupper}
Let $G$ and $H$ be arbitrary connected graphs, with $k=|V(H)|$, and $f:V(G)\rightarrow V(H)$ be an arbitrary non-constant function. Then, $\sum_{u,v}d(f(u),f(v))^2d_ud_v\leq \frac{\vol{G}^2\diam{H}^2}{2}\left(1-\frac{1}{|V(H)|}\right)$.
\end{thm}

\begin{proof}
Noting that for $i\neq j\in V(H)$, we have $d(i, j)\leq \diam{H}$, and following the technique and notation in the proof of Theorem \ref{denomlower}, we obtain

\begin{eqnarray*}
\sum_{u,v}d(f(u),f(v))^2d_ud_v & \leq & \diam{H}^2\sum_{i\neq j}x_ix_j\\
& = & \diam{H}^2\left( \frac{\|x\|_1^2-\|x\|_2^2}{2}\right)\\
& \leq & \diam{H}^2\left(\frac{\vol{G}^2 -\frac{1}{k}\vol{G}^2}{2}\right)\\
& = & \frac{\vol{G}^2\diam{H}^2}{2}\left(1-\frac{1}{|V(H)|}\right).
\end{eqnarray*}
\end{proof}

\section{Bounds on $\lambda(G, H)$}

We begin by proving Theorem \ref{T:mainone}, in the following four theorems.

\begin{thm}\label{T:lowerbound}
Let $G$ be a graph and $X$ a metric space with $|X|>2$. Then $\lambda(G, X)=0$ if and only if $G$ is disconnected.
\end{thm}

\begin{proof}

First, suppose that $G$ is disconnected, so exists a partition of $V$ into sets $V_1$ and $V_2$ such that $e(V_1,V_2)=0$, and $|V_1|, |V_2|>0$. Let $a,b\in X$ with $a\neq b$. Define a function $f:V\rightarrow X$ by
\[ f(v) = \left\{ \begin{array}{ll} a & v\in V_1\\ b & v\in V_2\end{array}\right. .\] 

Clearly, by definition, $R_f=0$, and hence $0\leq \lambda(G, X)\leq R_f=0$.

For the other direction, suppose that $G$ is connected, with diameter $D$. Let $f:V\to X$ be a nonconstant function, and let $C=\max_{u\sim v} d(f(u), f(v))$. Let $u, v\in V$ with $u\nmtd v$. As $G$ is connected, there exists a path $u=u_0\sim u_1\sim\dots\sim u_k\sim u_{k+1}=v$. By the triangle inequality, we have $d(f(u), f(v))\leq \sum_{i=0}^k d(f(u_i), f(u_{i+1}))\leq DC$. Thus we have

\begin{eqnarray*}
R_f & \geq & \frac{\vol{G}}{\Delta^2}\frac{\sum_{u\mtd v} d(f(u), f(v))^2}{\sum_{u\mtd v} d(f(u), f(v))^2+\sum_{u\nmtd v} d(f(u), f(v))^2}\\
 & \geq & \frac{\vol{G}}{\Delta^2}\frac{\sum_{u\mtd v} d(f(u), f(v))^2}{\sum_{u\mtd v} d(f(u), f(v))^2+\left({n\choose 2}-m\right)(DC)^2}\\
 & \geq & \frac{\vol{G}}{\Delta^2}\left(1-\frac{\left({n\choose 2}-m\right)(DC)^2}{\sum_{u\mtd v} d(f(u), f(v))^2+\left({n\choose 2}-m\right)(DC)^2}\right)
\end{eqnarray*}

Let $S_G = \left({n\choose 2}-m\right)D^2$, and note that this constant is independent of $f$. Moreover by definition, $\sum_{u\mtd v}d(f(u), f(v))^2\geq C^2$. Therefore,

\begin{eqnarray*}
\frac{\left({n\choose 2}-m\right)(DC)^2}{\sum_{u\mtd v} d(f(u), f(v))^2+\left({n\choose 2}-m\right)(DC)^2} &\leq & \frac{S_GC^2}{C^2+S_GC^2}\\
& = & \frac{S_G}{1+S_G}\leq1.
\end{eqnarray*}
Therefore, we have that for any function $f$, 

\[ R_f \geq \frac{\vol{G}}{\Delta^2}\left(1-\frac{S_G}{1+S_G}\right)>0,\]
and thus $\lambda(G, X)>0$ for any connected graph $G$.
\end{proof}

Note moreover that the proof technique yields the following immediate corollary.

\begin{cor}
If $G$ is a connected graph with diameter $D$, and $X$ is any metric space, then \[\lambda(G, X)\geq\frac{\vol{G}}{\Delta^2(1+S_G)},\] where $S_G=(1-\rho){n\choose 2}D^2$.
\end{cor}

We note that if $G$ is the complete graph, we obtain equality in the above bound. Indeed, if $G=K_n$, then for any function $f:V(G)\to X$, we have that $\sum_{u\mtd v} d(f(u), f(v))^2 = \sum_{u, v\in V(G)} d(f(u), f(v))^2$, and hence $R_f = n(n-1)/(n-1)^2=n/(n-1)$, regardless of the metric space into which we embed. In fact, this is the largest possible value that $\lambda(G, X)$ can take.

\begin{thm} Suppose that $G$ is a connected graph. Then, for any metric space $X$ where $|X|\geq 2$, we have $\lambda(G,X)\leq\frac{n}{n-1}$.
\end{thm}

\begin{proof}
Suppose that $\lambda >\frac{n}{n-1}$. Then, for all $f:V(G)\rightarrow X$, $R_f>\frac{n}{n-1}$. Equivalently, 
\[(n-1)\vol{G}\sum_{u\mtd v}d(f(u),f(v))^2>n\sum_{u,v}d(f(u),f(v))^2d_ud_v\]
Fix an arbitrary vertex $w$ of $G$ of minimal degree $\delta$ and define $f:V(G)\rightarrow X$ as mapping every vertex except $w$ to $a\in X$ and mapping $w$ to $b\in X$, where $d_X(a,b)=\epsilon$. Plugging this function into the inequality yields:
\[(n-1)\vol{G}\delta\epsilon^2>n\sum_{v\neq w}\epsilon^2\delta d_v=n\delta \epsilon^2(\vol{G}-\delta)\]
\[\implies \vol{G}<n\delta\]
which is a contradiction. 
\end{proof}
We have seen already that the complete graphs achieve this bound. Next, we see what can be learned about $G$ from knowing that $\lambda(G,X)=\frac{n}{n-1}$. 

\begin{thm}
Suppose that for some $G$, $\lambda(G,X)=\frac{n}{n-1}$. Then, $G$ is complete.
\end{thm}

\begin{proof}
The assumption means that $\inf_fR_f=\frac{n}{n-1}$. Plugging in the same function from the proof of Th. 8.1 yields:
\[\frac{\vol{G}\delta\epsilon^2}{\sum_{v\neq w}\epsilon^2\delta d_v}\geq\frac{n}{n-1}\implies \frac{\vol{G}}{\vol{G}-\delta}\geq\frac{n}{n-1}\implies \vol{G}\leq n\delta\]
But this implies that $G$ is $\delta-$regular since $\delta$ is the smallest degree. 
\\By assumption, we know that for all $f:V(G)\rightarrow X$, $\frac{n}{n-1}\leq\frac{\vol{G}\sum_{u\mtd v}d(f(u),f(v))^2}{\sum_{u,v}d(f(u),f(v))^2d_ud_v}=\frac{n\sum_{u\mtd v}d(f(u),f(v))^2}{\delta\sum_{u,v}d(f(u),f(v))^2}$. Select $p,q\in V(G)$ such that $p\mtd q$. Define $f:V(G)\rightarrow X$ by mapping all vertices in $G$ to $a\in X$ except for $p,q$, which get mapped to $b\in X$, where $d_X(a,b)=\epsilon$. Then, we get
\[\frac{n}{n-1}\leq\frac{2n(\delta-1)\epsilon^2}{2\delta (n-2)\epsilon^2}\implies\frac{\delta}{n-1}\leq\frac{\delta-1}{n-2}\implies \delta=n-1\implies G\text{ is complete.}\]
\end{proof}

We now turn to the proof of Theorem \ref{T:relatetoR}.

\begin{proof}[Proof of Theorem \ref{T:relatetoR}]
Suppose that $X$ is a finite metric space. Let $c, C$ be the constants guaranteed by Corollary \ref{C:Bourgain}, and let $f:X\to \R$ be the function guaranteed by the same corollary. Take $g:G\to X$ to be any nonconstant function. Then we obtain

\begin{eqnarray*}
R_g(G, X) & = & \frac{\vol{G}\sum_{u\mtd v}d_X(g(u), g(v))^2}{\sum_{u, v}d_X(g(u), g(v))^2d_ud_v}\\
& \geq & \frac{\vol{G}\sum_{u\mtd v} \frac{1}{C^2\log^4|X|}|f\circ g(u)-f\circ g(v)|^2}{\sum_{u, v}\frac{\log^2|X|}{c^2}|f\circ g(u)-f\circ g(v)|^2d_ud_v}\\
& = & \frac{c^2}{C^2\log^2|X|}R_{f\circ g}(G, \R)\\
& \geq & \frac{c^2}{C^2\log^2|X|}\lambda(G, \R).
\end{eqnarray*}
As this bound holds for all $g:G\to X$, taking the infimum yields the result.
\end{proof}

\subsection{Asymptotic lower bounds on $\lambda$}
Here, we investigate how quickly $\lambda$ can decrease to 0. We first prove a naive lower bound, and show that asymptotically this is best possible.

\begin{thm}\label{naivelower}
If $G$ is a connected graph on $n$ vertices, and $H$ is a graph on $k$ vertices with diameter $D_H$, then $\lambda(G,H)\geq\frac{2k}{D_H^2\vol{G}(k-1)}$.
\end{thm}

\begin{proof} First, we note that for any function $f:V(G)\to V(H)$, with $f$ nonconstant, we must have \[\sum_{u\mtd v}d(f(u), f(v))^2\geq 1.\] Using this together with the bound found in Theorem \ref{denomupper}, we obtain that for any $f$,
\[R_f \geq \frac{2\vol{G}}{D_H^2\vol{G}^2(1-1/k)} = \frac{2k}{D_H^2\vol{G}(k-1)}.\]
\end{proof}

Note that as $\vol{G}\leq n^2$, this result implies that for any graph family $\mathcal{G}$ and fixed graph $H$, the eigenvalues of $G_n\in \mathcal{G}$ with respect to $H$ decay no more rapidly than order $1/n^2$. As the next example shows, this is the optimal order of decay.

\begin{exa}
Construction a dumbbell graph $G=G_n$ by taking two copies of $K_{n/2}$ and attaching single edge between them. Let $H$ be $K_2$. Define a function $f:V(G)\to V(H)$ by mapping the vertices in the two dumbbells to opposite vertices in $H$. Then we obtain 
\[\lambda\leq\frac{\vol{G}\sum_{u\mtd v}d(f(u),f(v))^2}{\sum_{u,v}d(f(u),f(v))^2d_ud_v}=\frac{n(n/2-1)+2}{(n/2-1)^4 + n(n/2-1)^2 + (n/2)^2}\]

On the other hand, the lower bound given by Theorem \ref{naivelower} in this case is $\frac{4}{n(n/2-1)+2}$. Note that both bounds here are order $1/n^2$, and indeed, the constant is also the same; that is, both bounds decay as $8/n^2$. Thus, the bound given in Theorem \ref{naivelower} is asymptotically best possible.
\end{exa}
Now we turn our attention to regular graphs. We first note the following naive bound for $\lambda(G, H)$ for regular graphs
\begin{thm}\label{naivelowerreg}
If $G$ is a $d$-regular graph on $n$ vertices, and $H$ is a graph on $k$ vertices with diameter $D_H$, then $\lambda(G,H)\geq\frac{2}{(n-1)dD_H^2}$.
\end{thm}

\begin{proof}
As $G$ is regular, note that the denominator of $R_f$ may be written as $d^2\sum_{u, v} d(f(u), f(v))^2$. Hence, we have that for any nonconstant function $f:V(G)\to V(H)$,
\[ R_f\geq \frac{\vol{G}}{d^2{n\choose 2}D_H^2} = \frac{2}{d(n-1)D_H^2}.\]
\end{proof}

\begin{exa}
Construct a ``regularized dumbbell" graph as follows. First, take two copies of $K_{n/2}$. In each copy, select two vertices and delete the edge between them. Add two new edges between the two copies of $K_{n/2}\backslash\{e\}$ by connecting the each endpoint of the deleted edge in one copy to an endpoint of the deleted edge in the other copy.

Let $H=K_2$. Then, map the vertices in the dumbells to opposite vertices in $H$ as before. Then we obtain 
\[\lambda\leq\frac{\vol{G}\sum_{u\mtd v}d(f(u),f(v))^2}{\sum_{u,v}d(f(u),f(v))^2d_ud_v}\leq\frac{2n(n/2-1)}{(n/2-1)^2(n/2)^2}=\frac{8}{n(n/2-1)}=\frac{8}{nd}.\]

Note that the estimate given in Theorem \ref{naivelowerreg} is $2/nd$, and hence asymptotically, we have $\lambda(G, H)$ decays to 0 as quickly as possible.

\end{exa}


\section{Bounds relating $\lambda(G, H)$ to $\lambda(G, H')$}
Throughout this section, we will take the underlying metric space to be a graph $H$. We shall consider the effect to $\lambda(G, H)$ when changes are made to the graph $H$.

\begin{thm}\label{T:twographs}
Let $G$ be a connected graph, and let $H, H'$ be connected graphs on the same vertex set $V(H)$. Let $\lold=\lambda(G, H)$, and $\lnew=\lambda(G, H')$. Then, $\lnew\leq \frac{\Delta^2}{\delta^2}(1+S_G)\lold$, where $S_G$ is as in Theorem \ref{T:lowerbound}.
\end{thm}

\begin{proof}
If we can find a constant, $\beta$ such that for all nonconstant $f:V(G)\to V(H)$,
\begin{equation}\label{E:beta}R_f(G, H)\leq \beta R_f(G, H')\end{equation}
then $\lnew\leq \beta\lold$.

Fix such a function $f$. Let $d$ denote the distance function on $H$, and $d'$ the distance function on $H'$. As in the proof of Theorem \ref{T:lowerbound}, note that if $u\nmtd_G v$, then $d(f(u), f(v))^2\leq D^2 C^2$, where $C=\max_{x\mtd y} d(f(x), f(y))$. Consider

\begin{eqnarray*}
\frac{R_f(G, H')}{R_f(G, H)} & = & \frac{\left(\sum_{u\mtd v} d'(f(u), f(v))^2\right)\left(\sum_{u, v\in V} d(f(u), f(v))^2d_ud_v\right)}{\left(\sum_{u, v\in V} d'(f(u), f(v))^2d_ud_v\right)\left(\sum_{u\mtd v} d(f(u), f(v))^2\right)}
\end{eqnarray*}
\begin{eqnarray*}
&=&  \frac{\left(\sum_{u\mtd v} d'(f(u), f(v))^2\right)\left(\sum_{u\mtd v}d(f(u), f(v))^2d_ud_v\right) + \left(\sum_{u\mtd v} d'(f(u), f(v))^2\right)\left(\sum_{u\nmtd v}d(f(u), f(v))^2d_ud_v\right)}{\left(\sum_{u\mtd v}d'(f(u), f(v))^2d_ud_v\right)\left(\sum_{u\mtd v}d(f(u), f(v))^2 \right)+\left(\sum_{u\nmtd v}d'(f(u), f(v))^2d_ud_v\right)\left(\sum_{u\mtd v}d(f(u), f(v))^2 \right)}\\
& \leq & \frac{\Delta^2}{\delta^2}   \frac{\left(\sum_{u\mtd v} d'(f(u), f(v))^2\right)\left(\sum_{u\mtd v}d(f(u), f(v))^2\right) + \left(\sum_{u\mtd v} d'(f(u), f(v))^2\right)\left(\sum_{u\nmtd v}d(f(u), f(v))^2\right)}{\left(\sum_{u\mtd v}d'(f(u), f(v))^2\right)\left(\sum_{u\mtd v}d(f(u), f(v))^2 \right)+\left(\sum_{u\nmtd v}d'(f(u), f(v))^2\right)\left(\sum_{u\mtd v}d(f(u), f(v))^2 \right)}\\
& \leq &  \frac{\Delta^2}{\delta^2} \frac{\left(\sum_{u\mtd v} d'(f(u), f(v))^2\right)\left(\sum_{u\mtd v}d(f(u), f(v))^2\right) + \left(\sum_{u\mtd v} d'(f(u), f(v))^2\right)\left(\sum_{u\nmtd v}d(f(u), f(v))^2\right)}{\left(\sum_{u\mtd v}d'(f(u), f(v))^2\right)\left(\sum_{u\mtd v}d(f(u), f(v))^2 \right)}\\
& = & \frac{\Delta^2}{\delta^2}\left( 1+  \frac{\sum_{u\nmtd v}d(f(u), f(v))^2}{\sum_{u\mtd v}d(f(u), f(v))^2}\right)\\
& \leq & \frac{\Delta^2}{\delta^2}\left( 1+  \frac{\left({n\choose 2}-m\right)D^2C^2}{C^2}\right)\\
& = & \frac{\Delta^2}{\delta^2}\left(1+S_G\right).
\end{eqnarray*}

Hence, we may take $\beta =  \frac{\Delta^2}{\delta^2}(1+S_G)$ in equation (\ref{E:beta}), as desired. 
\end{proof}

By applying Theorem \ref{T:twographs} in both directions, we obtain the following immediate corollaries.

\begin{cor}
Let $G$ be a connected graph, and let $H, H'$ be two connected graphs on the same vertex set $V(H)$. Let all notation be as in Theorem \ref{T:twographs}. Then
\[\frac{\delta^2}{\Delta^2(1+S_G)}\lold\leq \lnew\leq \frac{\Delta^2(1+S_G)}{\delta^2}\lold.\]
\end{cor}

\begin{cor}
Let $G$ be a connected, $d$-regular graph with diameter $D$, and let $H, H'$ be connected graphs on vertex set $V(H)$. Then 
\[\frac{2}{n(n-1-d)D^2+2}\lold\leq\lnew\leq \left(1+\frac{n(n-1-d)D^2}{2}\right)\lold.\]
\end{cor}

In a similar way, we have the following bound.

\begin{thm}\label{T:twographssparse}
Let $G$ be a connected graph with $\vol{G}\geq 6$, and let $H, H'$ be connected graphs on the same vertex set $V(H)$, with $|V(H)|=k$. Let $\lold=\lambda(G, H)$, and $\lnew=\lambda(G, H')$, and let $D_H$ denote the diameter of $H$. Then, $\lnew\leq \frac{\Delta^2D_H}{\delta^2}(m(k-1)^2)\lold$.
\end{thm}

\begin{proof}
As in the proof of Theorem \ref{T:twographs}, we consider the ratio of $R_f(G, H')$ to $R_f(G, H)$ for an arbitrary nonconstant function $f:V(G)\to V(H)$. By Theorem \ref{denomlower}, we have

\begin{eqnarray*}
\frac{R_f(G, H')}{R_f(G, H)}& = & \frac{\left(\sum_{u\mtd v} d'(f(u), f(v))^2\right)\left(\sum_{u, v\in V} d(f(u), f(v))^2d_ud_v\right)}{\left(\sum_{u, v\in V} d'(f(u), f(v))^2d_ud_v\right)\left(\sum_{u\mtd v} d(f(u), f(v))^2\right)}
\end{eqnarray*}
\begin{eqnarray*}
&\leq&  \frac{\vol{G}^2D^2_H(1-1/k)}{2(\vol{G}-1)} \frac{\sum_{u\mtd v} d'(f(u), f(v))^2}{\sum_{u\mtd v} d(f(u), f(v))^2}\\
&\leq&  \frac{\vol{G}^2D^2_H(1-1/k)}{2(\vol{G}-1)}D^2_{H'}\\
&\leq&  \frac{3}{5}\vol{G}D^2_HD^2_{H'}.
\end{eqnarray*}
\end{proof}
Here, we have used the facts that given a fixed function $f$, any term in the sum $\sum_{u\mtd v}d(f(u),f(v))^2$ can increase by at most a factor of $D^2_{H'}$, that $1-1/k\leq 1$, and that $\frac{\vol{G}}{\vol{G}-1}\leq\frac{6}{5}$. As before, we obtain the following immediate corollary.

\begin{cor}
Let $G$ be a connected graph, and let $H, H'$ be connected graphs on the same vertex set $V(H)$, with $|V(H)|=k$. Let $\lold=\lambda(G, H)$, and $\lnew=\lambda(G, H')$, and let $D_H, D_{H'}$ denote the diameter of $H, H'$, respectively. Then
\[ \frac{5}{3\vol{G}D^2_HD^2_{H'}}\lold\leq\lnew\leq \frac{3\vol{G}D^2_HD^2_{H'}}{5}\lold.\]
\end{cor}

In comparing the bounds found in Theorems \ref{T:twographs} and \ref{T:twographssparse}, it seems that the stronger bound will be decided by the density of $G$. Indeed, as $S_G= ({n\choose 2}-m)D^2$, we note that if $m$ is quite large, Theorem \ref{T:twographs} will give a stronger bound, whereas if $m$ is quite small, the bound in Theorem \ref{T:twographssparse} will likely be stronger. 
Finally, we can improve the bound if a further constraint is made on $H$.
\begin{thm}
Take $H$ to be a complete graph and obtain $H'$ by removing one edge from $H$. Let $\lambda=\lambda(G,H)$ and $\lnew=\lambda(G,H')$. Then, 
\[\frac{\delta^2}{4\Delta^2}\lold\leq\lnew\leq\frac{4\Delta^2}{\delta^2}\lold\]
\end{thm}

\begin{proof}
\begin{eqnarray*}
\frac{\left(\frac{\vol{G}\sum_{u\mtd v} d'(f(u), f(v))^2}{\sum_{u, v\in V} d'(f(u), f(v))^2d_ud_v}\right)}{\left( \frac{\vol{G}\sum_{u\mtd v} d(f(u), f(v))^2}{\sum_{u, v\in V} d(f(u), f(v))^2d_ud_v}\right)} & = & \frac{\left(\sum_{u\mtd v} d'(f(u), f(v))^2\right)\left(\sum_{u, v\in V} d(f(u), f(v))^2d_ud_v\right)}{\left(\sum_{u, v\in V} d'(f(u), f(v))^2d_ud_v\right)\left(\sum_{u\mtd v} d(f(u), f(v))^2\right)}
\end{eqnarray*}
\begin{eqnarray*}
&=&  \frac{\left(\sum_{u\mtd v} d'(f(u), f(v))^2\right)\left(\sum_{u, v}d(f(u), f(v))^2d_ud_v\right)}{\left(\sum_{u, v}d'(f(u), f(v))^2d_ud_v\right)\left(\sum_{u\mtd v}d(f(u), f(v))^2 \right)}\\
&=&  \frac{\left(\sum_{u\mtd v} d'(f(u), f(v))^2\right)\left(\sum_{u, v}d_ud_v\right)}{m\sum_{u, v}d'(f(u), f(v))^2d_ud_v}\\
&\leq&  \frac{\binom{n}{2}\Delta^2}{m}\frac{\sum_{u\mtd v} d'(f(u), f(v))^2}{\sum_{u, v}d'(f(u), f(v))^2d_ud_v}\\
&\leq&  \frac{\binom{n}{2}\Delta^2}{m\delta^2}\frac{\sum_{u\mtd v} 4}{\sum_{u, v}1}=\frac{4\binom{n}{2}\Delta^2m}{m\delta^2\binom{n}{2}}=\frac{4\Delta^2}{\delta^2}.
\end{eqnarray*}
So, we have that $\lnew\leq\frac{4\Delta^2}{\delta^2}\lold$, and going in the opposite direction, we obtain the stated result.
\end{proof}
We now turn to the question of whether $\lambda$ is monotone in $H$.

\begin{thm}\label{T:subgraphsofH}
Let $H'$ be a connected subgraph of $H$. Then, $\frac{1}{D_{H'}^2}\leq\frac{\lambda(G,H')}{\lambda(G,H)}$.\end{thm}

\begin{proof}Note that for any pair of vertices $u, v$, we have $d_H(u, v)\leq d_{H'}(u, v)$. Moreover, $\frac{d_{H'}(u,v)}{d_H(u,v)}\leq D_{H'}$, and thus
\begin{eqnarray*}
\lambda(G,H)&=&\inf_{f:V(G)\rightarrow V(H)}\frac{\vol{G}\sum_{u\mtd v}d_H(f(u),f(v))^2}{\sum_{u,v}d_H(f(u),f(v))^2d_ud_v}\\
&\leq &\inf_{f:V(G)\rightarrow V(H')}\frac{\vol{G}\sum_{u\mtd v}d_{H'}(f(u),f(v))^2}{\sum_{u,v}d_H(f(u),f(v))^2d_ud_v}\\
&\leq& D_{H'}^2\inf_{f:V(G)\rightarrow V(H')}\frac{\vol{G}\sum_{u\mtd v}d_{H'}(f(u),f(v))^2}{\sum_{u,v}d_{H'}(f(u),f(v))^2d_ud_v}=D_{H'}^2\lambda(G,H'),
\end{eqnarray*}
as desired.
\end{proof}

It is straightforward to verify that $\lambda(K_{3,3}, K_2)=1=\lambda(K_{3,3}, K_4)$. Suppose that we label the vertices of $K_4$ as $0,1,2,3$. Let $K_4\ba\{e\}$ be the complete graph on four vertices, with the edge $e$ between 0 and 1 removed. Let $K_{3,3}$ have partition $v_1,v_2,v_3$ and $v_4,v_5,v_6$. Applying the function that maps $v_1,...,v_6$ in $K_{3,3}$ to $3,3,2,3,1,0$, respectively, shows that $\lambda(K_{3,3}, K_4\ba\{e\})\leq\frac{14}{15}<1$. Hence, as $K_2, K_4\ba\{e\}$ are both subgraphs of $K_4$, then $\lambda$ is not monotone in $H$.

However, we do obtain Theorem \ref{T:ktwoextreme} as a corollary:

\begin{cor}\label{C:Ktwo}
Let $H$ be a connected graph on $m>1$ vertices. Then $\lambda(G, H)\leq \lambda(G, K_2)$.
\end{cor}

\begin{proof}
Note that as $H$ is a connected graph, then $H'=K_2$ is a subgraph of $H$. Moreover, $D_{H'}=1$, and hence by Theorem \ref{T:subgraphsofH}, we have $\lambda(G, H)\leq \lambda(G, H')$.
\end{proof}

Hence, although we do not have monotonicity in $H$, we have that $H=K_2$ always provides an extreme value for $\lambda$. 

On the other hand, as every graph on $k$ vertices is a subgraph of $K_j$ for all $j>k$, we have the following corollary.

\begin{cor}\label{C:complete}
Let $H$ be a graph on $k$ vertices. Then $\lambda(G, H)\geq \frac{1}{D_H^2}\lambda(G, K_j)$ for all $j\geq k$. In particular, if $j>k$, then $\lambda(G, K_k)\geq \lambda(G, K_j)$.
\end{cor}
It is clear that once we can map $V(G)$ injectively into a complete graph $H$, then increasing the size of $H$ has no effect on $\lambda(G,H)$. So finally, we have:
\[\lambda(G,K_2)\geq\lambda(G,K_3)\geq ... \geq \lambda(G,K_n)=\lambda(G,K_{n+1})=...\]

\section{Bounds relating $\lambda(G, H)$ to $\lambda(G', H)$}

Here we consider the impact on $\lambda$ of adding or deleting edges to the graph $G$. We first consider what happens to $\lambda$ when one edge is added to $G$.

\begin{thm}
Let $G$ be a connected graph with $\vol{G}\geq 6$, and suppose that $G'$ is obtained from $G$ by adding one edge from $\overline{E(G)}$. Let $H$ be a graph with diameter $D_H$ and $|V(H)|=k$. Then \begin{equation}\label{upp}\frac{\lambda(G',H)}{\lambda(G,H)}\leq \left(1+\frac{2}{\vol{G}}\right)(1+D_H^2)\end{equation}
and
\begin{equation}\label{low}\frac{\lambda(G',H)}{\lambda(G,H)}\geq \left(1+\frac{2}{\vol{G}}\right)\left[\max\left\{\left(\frac{\vol{G}-1}{\vol{G}-1+D_H^2(2\vol{G}+1)}\right), \frac{1}{4}\right\} \right]\end{equation}
\end{thm}

\begin{proof}
We first consider the upper bound (\ref{upp}). Let $\{a, b\}\in \overline{E(G)}$, with $G'=G\cup\{a, b\}$. For any $f:V(G)\to V(H)$, we have
\[\sum_{u\mtd_{G'} v}d(f(u),f(v))^2 = \sum_{u\mtd_G v}d(f(u),f(v))^2+d(f(a),f(b))^2,\]
and $\vol{G'}=\vol{G}+2$. Hence,
\begin{eqnarray*}
\frac{\vol{G'}\sum_{u\mtd v\in E_{G'}}d(f(u),f(v))^2}{\vol{G}\sum_{u\mtd v\in E_G}d(f(u),f(v))^2} & = & \frac{\vol{G}+2}{\vol{G}}\frac{ \sum_{u\mtd_G v}d(f(u),f(v))^2+d(f(a),f(b))^2}{ \sum_{u\mtd_G v}d(f(u),f(v))^2}\\
& \leq & \left(1+\frac{2}{\vol{G}}\right)\left(1+\frac{D_H^2}{\sum_{u\mtd_G v}d(f(u),f(v))^2}\right)\\
& \leq & \left(1+\frac{2}{\vol{G}}\right)\left(1+{D_H^2}\right)
\end{eqnarray*}
Moreover, for any vertex $v\in V(G)$, we have $d_{G'}(v)\geq d_{G}(v)$, and hence we obtain 
\[\frac{R_f(G', H)}{R_f(G, H)}\leq \left(1+\frac{2}{\vol{G}}\right)\left(1+{D_H^2}\right),\]
yielding the upper bound.

For the lower bound, we will prove the two bounds separately. First, note that
\begin{eqnarray*}
\sum_{u, v}d(f(u), f(v))^2 d_{G'}(u)d_{G'}(v) &=& \sum_{u, v}d(f(u), f(v))^2 d_{G}(u)d_{G}(v)\\&&+\sum_{u}\left[d(f(u), f(a))^2+d(f(u), f(b))^2\right]d_G(u)+d(f(a), f(b))^2\\
& \leq &  \sum_{u, v}d(f(u), f(v))^2 d_{G}(u)d_{G}(v)+D_H^2(2\vol{G}+1)
\end{eqnarray*}
Hence, by Theorem \ref{denomlower}, we obtain
\begin{eqnarray*}
\frac{\sum_{u, v}d(f(u), f(v))^2 d_{G'}(u)d_{G'}(v)}{\sum_{u, v}d(f(u), f(v))^2 d_{G}(u)d_{G}(v)}&  \leq & 1+\frac{D_H^2(2\vol{G}+1)}{\vol{G}-1}\\
& \leq & 1+D_H^2\left(2+\frac{3}{\vol{G}-1}\right)
\end{eqnarray*}
Therefore, we have
\begin{eqnarray*}
\frac{R_f(G, H)}{R_f(G', H)} & \leq & \frac{\vol{G}}{\vol{G}+2}\left( 1+D_H^2\left(2+\frac{3}{\vol{G}-1}\right)\right)\frac{\sum_{u\mtd_{G}v} d(f(u), f(v))^2}{\sum_{u\mtd_{G'}v}d(f(u), f(v))^2}\\
& \leq & \left(\frac{\vol{G}}{\vol{G}+2}\right)\left( 1+D_H^2\left(2+\frac{3}{\vol{G}-1}\right)\right),
\end{eqnarray*}
yielding the first bound.

For the second bound, we simply note that 
\[\frac{\sum_{u,v}d(f(u),f(v))^2d_G(u)d_G(v)}{\sum_{u,v}d(f(u),f(v))^2d_{G'}(u)d_{G'}(v)}\geq \frac{\sum_{u,v}d(f(u),f(v))^2d_G(u)d_G(v)}{\sum_{u,v}d(f(u),f(v))^2(d_G(u)+1)(d_G(v)+1)}\]
and since $\frac{(d_u+1)(d_v+1)}{d_ud_v}\leq 4$, we get that $\frac{R_f(G',H)}{R_f(G,H)}\geq \frac{1}{4}\left(1+\frac{2}{\vol{G}}\right)$.

\end{proof}

We note that there are cases in which each of the two terms in the lower bound is larger, and hence both bounds can be useful. 

From the above theorem, it is unclear whether adding an edge to a graph $G$ will increase or decrease $\lambda(G, X)$. In fact, both possibilities can occur. To illustrate, we turn to the complete multipartite graph $K_{n, j}$. Here, $K_{n,j}$ represents the $j$-partite graph where each partition set $V_1,...,V_j$ has exactly $n$ vertices.

We shall consider the geometric eigenvalue $\lambda(K_{n,j}, K_2)$. Write $v_1 ,v_2$ as the vertices of $K_2$. Note that the only relevant pieces of information to evaluate $R_f(K_{n, j}, K_2)$ for a function $f$ are $|f^{-1}(v_1)\cap V_1|, |f^{-1}(v_1)\cap V_2|,..., |f^{-1}(v_1)\cap V_j|$. Indeed, if we denote these values by $x_i$, then we obtain
\[ R_f(K_{n, j}, K_2) = \frac{n^2j(j-1)\left(x_1(n-x_2)+(n-x_3)+...+(n-x_j)\right) }{n^2(j-1)^2(x_1+...+x_j)(nj-x_1-...-x_j) }= \frac{nj\sum_{i=1}^jx_i-\frac{2j}{j-1}\sum_{i<j}x_ix_j}{nj\sum_{i=1}^jx_i-\left(\sum_{i=1}^jx_j\right)^2}\]

\begin{thm} \label{T:bipartite}$\lambda(K_{n, j}, K_2)=1$ for all $n, j$.
\end{thm}

\begin{proof}
Using notation as above, define a function $f$ such that $x_i=1$ for all $i\in[j]$. Then by the above, $R_f=\frac{nj^2-\frac{2j}{j-1}\frac{j(j-1)}{2}}{nj^2-j^2}=1$. 

On the other hand, note that if $(x_1+...+x_j)^2\geq\frac{j}{j-1}\sum_{i<j}x_ix_j$ for any real numbers $x_1,...,x_j$, then we have that $R_f\geq 1$ for all $f$. Therefore, it must be that the given function $f$ achieves the infimum, and $\lambda(K_{n, j}, K_2)=1$. 
\\So, we only need to show that $(x_1+...+x_j)^2\geq\frac{j}{j-1}\sum_{i<j}x_ix_j$ holds, which by expanding the square is equivalent to demonstrating that $(j-1)x_1^2+...+(j-1)x_j^2- 2\sum_{i<j}x_ix_j\geq 0$. This is true since the left hand side equals $\sum_{i<j}(x_i-x_j)^2$, which is non-negative as it is a sum of squares.

\end{proof}

Using this result, we immediately have the following.

\begin{thm}There exists a pair of graphs $G, G'$, such that $G'$ is obtained from $G$ by adding one edge, and $\lambda(G, K_2)< \lambda(G', K_2)$.
\end{thm}

\begin{proof}
Let $G_1=K_{n, n}$, so as seen above, $\lambda(G_1, K_2)=1$. Moreover, $\lambda(K_n, K_2)=\frac{n}{n-1}>1$, and hence if we add the nonedges of $G_1$ sequentially, we will encounter a pair of graphs satisfying the condition.
\end{proof}

\begin{thm}
There exists a pair of graphs $G, G'$, such that $G'$ is obtained from $G$ by adding one edge, and $\lambda(G, K_2)> \lambda(G', K_2)$.
\end{thm}

\begin{proof}
Let $G=K_{n, n}$, with bipartition sets $V_1$ and $V_2$, as above, and $n\geq 3$. Let $x, y\in V_1$, and let $G' = G\cup\{xy\}$. Define a function $f:V(G')\to V(K_2)$ as follows. Let $z_1\in V_1$, with $z_1\neq x, y$, and $z_2\in V_2$. Let $f(z_1)=f(z_2)=v_1$, and $f(v)=v_2$ for all other $v\in V(G)$. Then we have
\begin{eqnarray*}
R_f & = & \frac{((2n-2)n+2(n+1))(2(n-1))}{n^2(2(n-3+n-1))+n(n+1)(4)}\\
& = & \frac{(n^2+1)(n-1)}{n(n^2-n+1)}\\
& = & \frac{n^3-n^2+n-1}{n^3-n^2+n}<1.
\end{eqnarray*}
Hence, $\lambda(G', K_2)<1=\lambda(G, K_2)$.
\end{proof}

The two results above suffice to show that $\lambda(G, H)$ is not monotone in $G$ when $H=K_2$. However, the proof also works for $\lambda(G,\mathbb{R})$ (instead of mapping into $K_2$, map to the real numbers 0 and 1).
\begin{lemma}
$\lambda(K_{n,n},K_k)= 1$ for all $n,k$.
\end{lemma}

\begin{proof}
Note that by Theorem \ref{T:bipartite} and Corollary \ref{C:complete}, it suffices to show that $\lambda(K_{n, n}, K_k)\geq 1$ for all $n,k$.

For all $i\in [k]$, and for a function $f:V(G)\rightarrow K_k$, define $x_i=|f^{-1}(v_i)\cup V_1|$ and $y_i=|f^{-1}(v_i)\cup V_2|$. We consider $x$ and $y$ to be the vectors of these numbers, and we know that $||x||_1=||y||_1=n$. For a fixed $x$, $||x||_2^2-x^Ty$ is minimized when $y$ is a multiple of $x$. Since $y$ cannot equal $cx$ for any $c\neq 1$ (due to the 1-norm constraint on $y$), we know that $||x||_2^2-x^Ty$ is minimized when $y=x$, and this is lower bounded by 0. 
\\Thus, we have:
\[0\leq x_1^2+...+x_k^2-x_1y_1-...-x_ky_k\]
\[=x_1(n-y_1)+...+x_k(n-y_k)-x_1(n-x_1)-...-x_k(n-x_k)\]
\[=x_1(y_2+...+y_k)+...+x_k(y_1+...+y_{k-1})-x_1(x_2+...+x_k)-...-x_k(x_1+...+x_k)\]
Then, we have that:
\[\beta :=\frac{x_1(y_2+...+y_k)+...+x_k(y_1+...+y_{k-1})}{x_1(x_2+...+x_k)+...+x_k(x_1+...+x_k)}\geq 1\]
For any $f:V(G)\rightarrow K_k$, we have:
\[R_f=2\frac{x_1(y_2+...+y_k)+...+x_k(y_1+...+y_{k-1})}{x_1(y_2+...+y_k)+...+x_k(y_1+...+y_{k-1})+x_1(x_2+...+x_k)+...+x_k(x_1+...+x_k)}\]
\[\implies R_f^{-1}=\frac{1}{2}\left(1+\frac{1}{\beta}\right)\leq 1\implies R_f^{-1}\leq 1\implies R_f\geq 1\]
\end{proof}
So now we can extend our proof of non-monotonicity to $H=K_k$. 
\begin{cor}
$\lambda(G,H)$ is not monotonic in adding an edge to $G$, whenever $H$ is a complete graph. 
\end{cor}

\begin{proof}
By lemma 6.8 and corollary 6.6, we know that $\lambda(K_{n,n},K_k)=1$. But then the same example as in theorem 6.4 also works to disprove monotonicity in the generalized case, as $\lambda(G',K_k)\leq\lambda(G',K_2)<1$. 
\end{proof}
We note that a graph being bipartite is not equivalent to $\lambda(G,H)$ equaling 1. Of course, a disconnected graph can be bipartite, but $\lambda$ will equal 0. Similarly, if one wants a connected counterexample, it is easy to check that $\lambda(P_3,K_2)=\frac{4}{3}$, where $P_3$ is the path on three vertices. Conversely, our computer simulation tells us that $\lambda(G,K_2)=1$ where $G$ is the nonbipartite graph that is formed by taking $K_4$ and deleting two edges that touch a common vertex. 
\\We have established that adding an edge to $G$ can decrease $\lambda(G,H)$, despite the improvement in connectivity. In fact, we have something stronger. If we have a $k$-regular graph, we can add enough edges to it so as to make it $k+1$-regular and still have a decrease in lambda. A variant of this has been considered in \cite{mendel2013nonlinear}, and a similar bound developed.

\begin{thm}
Let $G'$ be a $k+1$-regular supergraph of $k$-regular graph $G$. Then, \[\frac{k}{k+1}\leq\frac{\lambda(G',H)}{\lambda(G,H)}\leq\frac{k}{k+1}\left(1+\frac{nD_H^2}{2}\right)\]
\end{thm}

\begin{proof}
First, note that \[\frac{R_f(G',H)}{R_f(G,H)}=\frac{\vol{G'}k^2\sum_{u\mtd v\in G'}d'(f(u),f(v))^2\sum_{u,v}d(f(u),f(v))^2}{\vol{G}(k+1)^2\sum_{u\mtd v\in G}d(f(u),f(v))^2\sum_{u,v}d'(f(u),f(v))^2}=\frac{k}{k+1}\frac{\sum_{u\mtd v\in G'}d'(f(u),f(v))^2}{\sum_{u\mtd v\in G}d(f(u),f(v))^2}\]
For the lower and upper bounds, we use:
\[1\leq\frac{\sum_{u\mtd v\in G'}d(f(u),f(v))^2}{\sum_{u\mtd v\in G}d(f(u),f(v))^2}\leq\frac{\sum_{u\mtd v\in G}d(f(u),f(v))^2+\sum_{u\mtd v\in E(G')-E(G)}d(f(u),f(v))^2}{\sum_{u\mtd v\in G}d(f(u),f(v))^2}\leq 1+\frac{nD_H^2}{2}\]
\end{proof}

\begin{exa}
For $n$ even, let $G=K_{n,n}$, the complete, balanced, bipartite graph on 2n vertices. We already know that this is a $n$-regular graph with $\lambda(G,K_2)=1$. Denote the vertices on the left partition as $v_1,v_2,...,v_n$ and the vertices on the right as $v_{n+1},...,v_{2n}$. Then, create $G'$ by adding the edges $v_1\mtd v_2$, $v_3\mtd v_4,...,v_{n-1}\mtd v_n,v_{n+1}\mtd v_{n+2},...,v_{2n-1}\mtd v_{2n}$. Then, $G'$ is $n+1$-regular and applying the function that maps $v_1,v_2,v_{n+1},v_{n+2}$ to one vertex in $K_2$ and the other vertices to the other vertex in $K_2$, we obtain that $\lambda(G',K_2)\leq\frac{n}{n+1}$, which meets the lower bound in theorem 6.10.
\end{exa}
Perhaps a natural question to ask at this point is how many of the nonedges can be added to a graph so that $\lambda(G,K_2)$ still decreases. We consider this problem by trying to maximize the number of added edges divided by the total number of nonedges in $G$, for some graph family. One goal is to create a graph family so that for each member of the family, some constant fraction of the edges can be added in such a way that ambda decreases. The following is the closest we could come to that.

\begin{thm}
There exists a graph family $\{\mathcal{G}_n\}$ such that for each $G\in \{\mathcal{G}_n\}$, there exists a graph $G'$ such that $\lambda(G',K_2)<\lambda(G,K_2)$ and $G'$ is obtained from $G$ by adding a set of edges to $G$ of size $O\left(\frac{1}{n^{\epsilon}}|\overline{E}|\right)$, for any $\epsilon>0$.
\end{thm}

\begin{proof}
The construction is to take a complete, balanced, bipartite graph $G=K_{n, n}$, so that we have $\lambda(G,K_2)=1$. On each side of the partition $V_1$, $V_2$, let there be a set of $n^{1-\epsilon/2}$ red vertices and a set of $n-n^{1-\epsilon/2}$ blue vertices. Create $G'$ by adding edges to $G$ so that the red vertices in $V_1$ form a clique and the red vertices in $V_2$ form another clique. 
\\We have added $n^{1-\epsilon/2}(n^{1-\epsilon/2}-1)$ edges. Thus, we have added $\frac{n^{1-\epsilon/2}(n^{1-\epsilon/2}-1)}{n(n-1)}=O\left(\frac{1}{n^{\epsilon}}|\overline{E}|\right)$ of the missing edges.
\\Fix a function $f:V(G)\rightarrow K_2$, where $f$ maps red vertices to one vertex and blue vertices to the other. All that remains is to see that $R_f<1$. We have:
\begin{eqnarray*}R_f&=&\frac{\vol{G'}\sum_{u\mtd v}d(f(u),f(v))^2}{\sum_{u,v}d(f(u),f(v))^2d_ud_v}\\&=&\frac{4(n^2+n^{1-\epsilon/2}(n-n^{1-\epsilon/2}))n^{1-\epsilon/2}(n-n^{1-\epsilon/2})}{4n^{1-\epsilon/2}(n-n^{1-\epsilon/2})n(n+n^{1-\epsilon/2}-1)}\\
&=&\frac{n^{4-\epsilon/2}-n^{4-\epsilon}+n^{4-3\epsilon/2}-n^{4-2\epsilon}-n^{3-\epsilon}+n^{3-3\epsilon/2}  }{n^{4-\epsilon/2}-n^{4-3\epsilon/2}-n^{3-\epsilon/2}+n^{3-\epsilon}}\\
& \sim & \frac{n^{3\eps/2}-n^\eps+n^{\eps/2}-1}{n^{3\eps/2}-n^{\eps/2}}<1\end{eqnarray*}
for small enough $\epsilon$ and large enough $n$.
\end{proof}

Finally, recall from Theorem \ref{T:mainone} that $\lambda(G, H)<\lambda(K_n, H) = \frac{n}{n-1}$ for all non-complete graphs $G$ on $n$ vertices. We end with an examination of the ``nearly'' complete graph $K_n\ba\{e\}$, a graph with exactly one nonadjacent pair of vertices.

\begin{thm}
For $n\geq 3$, $\lambda(K_n\ba\{e\},K_2)\geq 1$.  
\end{thm}
\begin{proof}
Denote the vertices of $K_n\ba\{e\}$ as $v_0,...,v_{n-1}$ and let there be no edge between $v_0$ and $v_1$. Write $V(K_2)=\{0, 1\}$. It suffices to consider two cases, depending on whether $v_0$ and $v_1$ map to the same vertex in $K_2$. 

First, consider a function $f$ with $f(v_0)=f(v_1)$. Let $x$ denote the number of vertices other than $v_0, v_1$ with image $0$ and $y$ the number of vertices other than $v_0,  v_1$ with image $1$. Without loss of generality, suppose that $f(v_0)=f(v_1)=0$. Then
\[R_f=\frac{(n(n-1)-2)(xy+2y)}{xy(n-1)^2+2y(n-1)(n-2)}=\frac{xy(n+1)(n-2)+2y(n+1)(n-2)}{xy(n-1)^2+2y(n-1)(n-2)}\geq 1.\]
The final inequality follows from the fact that $(n+1)(n-2)$ is greater than $(n-1)^2$ whenever $n\geq 3$. 

For the second case, assume that $f(v_0)=0$, $f(v_1)=1$, and $x$ and $y$ are defined as before. Then
\[R_f=\frac{(n(n-1)-2)(xy+x+y)}{x(n-1)(n-2)+(n-2)^2+xy(n-1)^2+y(n-1)(n-2)}=\frac{xy(n+1)(n-2)+(n+1)(n-2)^2}{xy(n-1)^2+(n-1)(n-2)^2+(n-2)^2}.\]
The second equality follows from the fact that $x+y=n-2$. Dividing the numerator and denominator by $(n-1)^2$, it is clear that this ratio is at least 1, concluding the proof.
\end{proof}

\begin{cor}
$\lim_{n\rightarrow\infty}\lambda(K_n\ba\{e\},K_2)=1$
\end{cor}
\begin{proof}
Combine the previous theorem with the result from Theorem \ref{T:mainone}, so that \[\lambda(K_n\ba\{e\},K_2)< \lambda(K_n,K_2)=\frac{n}{n-1}\rightarrow 1.\]
\end{proof}
Note that since $\lambda(K_3\ba\{e\},K_2)=\frac{4}{3}\geq 1$, transforming a graph from $K_n\ba\{e\}$ to $K_{n+1}\ba\{e\}$ can decrease lambda. This means that taking a graph and adding a new vertex which is connected to all previous vertices can decrease lambda, although intuitively it might appear that this should produce a graph that is ``better'' connected than the original. We note that this operation can also increase $\lambda$; if $G=K_{2, 2}$, for example, then adding a vertex adjacent to all four original vertices will produce a graph with a larger $\lambda$ value with respect to $K_2$.

\bibliographystyle{abbrv}  
 \bibliography{bib_items}

\end{document}